\documentclass[11pt]{amsart}
\usepackage{preamble}

\usepackage{tikz-cd}

\newcommand{\Ex}{\operatorname{Ex}}

\newcommand{\Mov}{\operatorname{Mov}}

\newcommand{\Fix}{\operatorname{Fix}}

\begin{document}
\title[A characterization of Fano type varieties]{A characterization of Fano type varieties}


\begin{abstract}
We prove a characterization of Fano type varieties.
\end{abstract}

\author{Yiming Zhu}
\address{School of Mathematical Science, University of Science and Technology of China, Hefei 230026, P.R.China.} \email{zym18119675797@gmail.com}

\maketitle

\setcounter{tocdepth}{1}
\tableofcontents

\section{Introduction}

Let $X$ be a normal variety. We say that $X$ is of klt type if there is an effective $\Qq$-divisor $\Delta$ such that $(X,\Delta)$ is klt. Let $X$ be a projective normal variety. We say that $X$ is of Fano type if there is an effective $\Qq$-divisor $\Delta$ such that $(X,\Delta)$ is klt and $-(K_X+\Delta)$ is ample.  

In \cite{zhuang2024direct}, Zhuang proved that pure images of klt type singularities are of klt type. The following klt type criterion is crucial in his proof.

\begin{lem}\cite[Lemma 2.4]{zhuang2024direct}\label{lem:zhuang}
Let $X$ be a normal variety. Then $X$ is of klt type if and only if 
\begin{enumerate}
    \item the anticanonical ring $R(X,-K_X):=\bigoplus_{m\geq0}O_X(-mK_X)$ is finitely generated over $O_X$, and
    \item $X':=\proj _XR(X,-K_X)$ is klt.
\end{enumerate}
\end{lem}

In this note, we prove the following theorem, which is a global version of \ref{lem:zhuang}. 

\begin{thm}\label{thm:main}
Let $X$ be a projective normal variety. Then $X$ is of Fano type if and only if 
\begin{enumerate}
    \item $-K_X$ is big,
    \item the anticanonical ring $R(X,-K_X):=\bigoplus_{m\geq0}H^0(X,-mK_X)$ is finitely generated, and 
    \item $Y:=\proj_{\Cc} R(X,-K_X)$ is klt. 
\end{enumerate}
\end{thm}

Varieties of Fano type are not necessarily $\Qq$-Gorenstein. Hence Theorem \ref{thm:main} generalizes \cite[Theorem 1.1]{CasciniGongyo2013anticanonical} or \cite[Theorem 1.1]{ChoiHwangPark2015Factorization}, where the $\Qq$-Gorenstein assumption on the variety is required.

We discuss section rings for Weil divisors in Section 2, and prove Theorem \ref{thm:main} in Section 3.

\subsection*{Acknowledgments}
The author thanks Professor Chen Jiang for providing a simpler proof of Theorem \ref{thm:main}. The author acknowledges Professor Mao Sheng for the financial support through the CAS Project for Young Scientists in Basic Research (Grant No. YSBR-032).

\subsection*{Notation and conventions}
We work over the field of complex numbers. 

Let $f:X\to Y$ be a projective morphism with connected fibers between normal varieties. Let $D$ be an effective $\Qq$-divisor on $X$. There are effective $\Qq$-divisors $D^h$ and $D^v$ such that $D=D^h+D^v$ and a component $P$ of $D$ dominates $Y$ if and only if $P\subset \Supp D^h$. We call such $D^h$ and $D^v$ the horizontal part and the vertical part of $D$, respectively. We say that $D$ is vertical over $Y$ if $D=D^v$. If $D$ is vertical over $Y$, we say that $D$ is exceptional over $Y$ if $\codim f(D)\geq2$, and very exceptional over $Y$ if for any prime divisor $P\subset Y$, there is a prime divisor $Q\subset X$ such that $f(Q)=P$ but $Q\not\subset\Supp D$.

For a Weil divisor $D$ on a projective normal variety $X$, we denote by $f_D$ the rational map associated to the linear system $|D|$. For a linear system $|D|$, we denote by $\Mov|D|$ and $\Fix|D|$ its moving part and fixed part, respectively.


\section{On Weil divisors whose section rings are finitely generated}

In this section, unless otherwise specified, a divisor refers to a Weil divisor.

\begin{proof}[\textbf{Setting}]\label{setting}
Let $X$ be a projective normal variety and let $D$ be a (non-necessarily Cartier) divisor on $X$, such that $R(X,D):=\bigoplus_{m\geq0}H^0(X,mD)$ is finitely generated and is generated in degree one. Let $Y:=\proj R(X,D)$. Let $f=f_{D}:X\dashrightarrow\Pp^{N}$ be the rational map associated to the linear system $|D|$. Let $\pi:X'\to X$ be a resolution of singularities, such that the rational map $f':X'\dashrightarrow Y$ is a morphism. Let $\Ex(\pi)$ be the codimension one part of the exceptional locus of $\pi$.    
\end{proof}

This section aims to generalize the following well-known theorem to Weil divisors.

\begin{thm}\label{IitakaCartier}
In the \textbf{Setting}. Assume that $D$ is Cartier. Then $Y$ is normal, $Y=f'(X')$, $f'$ has connected fibers, the rational maps $f_{mD}$ are birationally equivalent to $f$ for all $m>0$, and there is an ample and free Cartier divisor $A$ on $Y$, such that $\Mov|m\pi^*D|=|mf'^*A|$ for all $m>0$.\end{thm}

We start with the following Lemma.

\begin{lem}\label{lem:key}In the \textbf{Setting}.
\begin{enumerate}
\item There is a divisor $D'$ on $X'$ such that \begin{align}\label{one-one}
\pi_*D'=D \text{ and } \pi_*:H^0(D')\cong H^0(D). 
\end{align}
\item Among all divisors satisfying \ref{one-one}, there is a minimal one.
\item If $D'$ is a divisor satisfying \ref{one-one} and $E$ is an effective $\pi$-exceptional divisor, then $D'+E$ also satisfies \ref{one-one}.
\item If $D'$ is a divisor satisfying \ref{one-one}, then $\pi_*:H^0(mD')\cong H^0(mD)$ for all $m\geq0$.
\end{enumerate}

\end{lem}
\begin{proof}
(1) Let $\{s_i\}$ be basis of $H^0(D)\subset\Cc(X)$. Then a divisor $D'$ with $\pi_*D'=D$ satisfies \ref{one-one} if and only if $\di_{X'}(s_i)+D'\geq0$ for all $i$. Hence $\pi^{-1}_*D+m\Ex(\pi)$ satisfies \ref{one-one} for all $m\gg0$.

(2) If $D_1$, $D_2$ are divisors satisfying \ref{one-one}, then $D_1\wedge D_2$ also satisfies \ref{one-one}. 

(3) It is clear from the argument in (1).

(4) It follows from the commutative diagram
\begin{center}
\begin{tikzcd}
\text{Sym}^mH^0(D') \arrow[d, "\cong"] \arrow[r] & H^0(mD') \arrow[d, "\pi_*", hook] \\
\text{Sym}^mH^0(D) \arrow[r, two heads]            & H^0(mD).      
\end{tikzcd}  
\end{center}
\end{proof}

\begin{thm}\label{IitakaWeil}
In the \textbf{Setting}. We have $Y$ is normal, $Y=f'(X')$, $f'$ has connected fibers, the rational maps $f_{mD}$ are birationally equivalent to $f$ for all $m>0$, and the following holds.
\begin{enumerate}
\item There is an ample and free Cartier divisor $A$ on $Y$ such that
if $D'$ is a divisor satisfying  \ref{one-one}, then $\Mov|mD'|=|mf'^*A|$ for all $m>0$. 
    
\item If $D'$ is a divisor satisfying  \ref{one-one}, then the $f'$-vertical part of $\Fix|D'|$ is very exceptional over $Y$.

\item The $f'$-vertical part of $\Ex(\pi)$ is very exceptional over $Y$.
\end{enumerate}

\end{thm}
\begin{proof}
By Lemma \ref{lem:key} and Theorem \ref{IitakaCartier}, all statements except (1), (2), and (3) are clear. 

Let $\Tilde{D}$ be the minimal one among all divisors satisfying \ref{one-one}. Write $|\Tilde{D}|=|M|+\tilde{F}$ as a sum of moving part and fixed part. By Lemma \ref{lem:key} and Theorem \ref{IitakaCartier}, there is an ample and free Cartier divisor $A$ on $Y$ such that $\Mov|m\Tilde{D}|=|mf'^*A|$ and $\Fix|m\Tilde{D}|=m\tilde{F}$ for all $m>0$. 

We prove (1) and (2) in this paragraph. Let $D'$ be a divisor satisfying \ref{one-one}. Since $\Tilde{D}$ is minimal, there is an effective divisor $E$ such that $D'=\Tilde{D}+E$. By Lemma \ref{lem:key}(4), we have $|mD'|=|m\Tilde{D}|+mE$. Hence $\Mov|mD'|=\Mov|m\Tilde{D}|=|f'^*mA|$. Let $F'$ be the $f'$-vertical part of $\Fix|D'|$. By $\Fix|m(M+F')|=mF'$ for all $m>0$ and \cite[Lemma 3.2]{birkar2012existence}, we have $F'$ is very exceptional over $Y$.

(3) Since $\tilde{D}+\Ex(\pi)$ satisfies \ref{one-one} and $\Ex(\pi)\leq \Fix|\tilde{D}+\Ex(\pi)|$, we have the $f'$-vertical part of $\Ex(\pi)$ is very exceptional over $Y$ by (2).
\end{proof}

\section{Proof of Theorem \ref{thm:main}}

\begin{lem}\label{lem:simple}
Let $(X,\Delta)$ be a projective klt pair with $-(K_X+\Delta)$ nef and big, then $X$ is of Fano type.     
\end{lem}
\begin{proof}
Since $-(K_X+\Delta)$ is big, there are an ample $\Qq$-divisor $A$ and an effective $\Qq$-divisor $E$ such that $-(K_X+\Delta)\sim_{\Qq}A+E$. If $0<\epsilon\ll1$, then $(X,\Delta+\epsilon E)$ is klt and $-(K_X+\Delta+\epsilon E)\sim_{\Qq}(1-\epsilon)(A+E)+\epsilon A$ is ample.   
\end{proof}

\begin{lem}\label{lem:simple2}
Let $f:X\to Y$ be a birational morphism between projective normal varieties. If $X$ is of Fano type, then $Y$ is of Fano type.         
\end{lem}
\begin{proof}
Let $\Delta$ and $A$ be effective $\Qq$-divisors such that $A$ is ample, $(X,\Delta+A)$ is klt and $K_X+\Delta+A\sim_{\Qq}0$. By $K_X+\Delta+A\sim_{\Qq}f^*(K_Y+f_*\Delta+f_*A)$, we have $(Y,f_*\Delta+f_*A)$ is klt. Since $f_*A$ is big, there are an ample effective $\Qq$-divisor $A'$ and an effective $\Qq$-divisor $E$ such that $f_*A\sim_{\Qq}A'+E$. If $0<\epsilon\ll1$, then $(Y,f_*\Delta+(1-\epsilon)f_*A+\epsilon E)$ is klt and $-(K_Y+f_*\Delta+(1-\epsilon)f_*A+\epsilon E)\sim_{\Qq}\epsilon A$ is ample.
\end{proof}

\begin{proof}[Proof of Theorem \ref{thm:main}]
\textbf{The ``only if'' part.} Since $X$ is of Fano type, we have $-K_X$ is big and $X$ is of klt type. By Zhuang's Lemma \ref{lem:zhuang}, we have $\bigoplus_{m\geq0}O_X(-mK_X)$ is finitely generated over $O_X$. Consider the birational morphism $\pi:X':=\proj_X\bigoplus_{m\geq0}O_X(-mK_X)\to X$. By \cite[Lemma 6.2]{KM98}, $\pi$ is small, $K_{X'}$ is $\Qq$-Cartier, and \begin{align}\label{one-one 2}
\pi_*O_{X'}(-mK_{X'})=O_X(-mK_X) \text{  for all } m\geq0.    
\end{align} Hence $X'$ is a $\Qq$-Gorenstein Fano type variety by \cite[Lemma 2.4]{Birkar2016FanotypeFibraiton}. By  \cite[Theorem 1.1]{CasciniGongyo2013anticanonical} and \ref{one-one 2}, we have $R(X,-K_X)\cong R(X',-K_{X'})$ is finitely generated and $Y=\proj R(X,-K_X)\cong \proj R(X',-K_{X'})$ is klt.

\textbf{The ``if'' part.} Let $r$ be a positive integer such that $R(X,-rK_X)$ is generated in degree one. Let $f:X\dashrightarrow\Pp^N$ be the rational map associated to the linear system $|-rK_X|$. Let $\pi:X'\to X$ be a resolution of singularities such that the rational map $f':X'\dashrightarrow \Pp^N$ is a morphism. Let $\Ex(\pi)$ be the codimension one part of the exceptional locus of $\pi$. By Lemma \ref{lem:key}, there is an effective and $\pi$-exceptional divisor $E$ on $X'$ such that \[\pi_*:|m(-rK_{X'}+E)|\to |-mrK_X|\text{ is one-one for all }m>0.\] 

Write $|-rK_{X'}+E|=|M|+F$ as a sum of moving part and fixed part. By Theorem \ref{IitakaWeil}, we have 
\begin{enumerate}
    \item $Y=f'(X')$,
    \item $f'$ has connected fibers,
    \item $M=f'^*A$ for an ample Cartier divisor $A$ on $Y$, and 
 \item the $f'$-vertical part of $F$ and $\Ex(\pi)$ are very exceptional over $Y$. 
\end{enumerate}
Since $-K_{X}$ is big and $f'$ has connected fibers, we have $f'$ is a birational morphism. Hence $F$ and $\Ex(\pi)$ are exceptional over $Y$. Hence, the rational map $f:X\dashrightarrow Y$ is a birational contraction, i.e., there are no $f^{-1}$-exceptional divisors on $Y$. By \[-rK_{Y}=f'_*(-rK_{X'})=f'_*(-rK_{X'}+E)=f'_*(M+F)=f'_*(f'^*A+F)=A,\] 
we have \begin{align}\label{discrep}
K_{X'}-\frac{1}{r}E+\frac{1}{r}F\sim_{\Qq}f^*K_Y.    
\end{align}

Let $P\subset X'$ be any prime divisor which is $f'$-exceptional and non $\pi$-exceptional. By \ref{discrep}, we have \[a(P,Y)=a(P,X',-\frac{1}{r}E+\frac{1}{r}F)=-\mult_P(-\frac{1}{r}E+\frac{1}{r}F)=-\mult_P\frac{1}{r}F\leq0.\]By \cite[Corollary 1.4.3]{BCHM}, there is a projective klt pair $(Z,\Delta_Z)$ and a birational morphism $g:Z\to Y$ such that $K_Z+\Delta_Z\sim_{\Qq}g^*K_Y$, and
the induced rational map $Z\dashrightarrow X$ is small. Since $-K_Y$ is ample, we have $-(K_Z+\Delta_Z)$ is nef and big. Hence $Z$ is of Fano type by Lemma \ref{lem:simple}. By \cite[Lemma 2.4]{Birkar2016FanotypeFibraiton}, $X$ is of Fano type.

\textbf{Another Proof of the ``if'' part (Provided by Chen Jiang).} By \cite{BCHM}, there is a projective $\Qq$-factorial terminal variety $Y$ and a birational morphism $\pi:X'\to X$ such that $K_{X'}$ is nef over $X$. 

We claim that $\pi_*:|-mK_{X'}|\to |-mK_X|$ is one-one for all $m\geq0$. Let $s$ be a rational function on $X$ such that $\di_X(s)-mK_X\geq0$. Then $\di_{X'}(s)-mK_{X'}$ is a $\pi$-anti-nef divisor with $\pi_*(\di_{X'}(s)-mK_{X'})=\di_X(s)-mK_X\geq0$. By the Negativity Lemma, we have $\di_{X'}(s)-mK_{X'}\geq0$. Hence, the claim follows. 

Since ${X'}$ is $\Qq$-Gorenstein, $R(Y,-K_{X'})\cong R(X,-K_X)$ is finitely generated and $\proj R(X',-K_{X'})\cong \proj R(X,-K_X)$ is klt, we have $X'$ is of Fano type by \cite[Theorem 1.1]{CasciniGongyo2013anticanonical}. By Lemma \ref{lem:simple2}, $X$ is of Fano type.
\end{proof}



\bibliographystyle{alpha}

\bibliography{bibfile}

\end{document}